\documentclass[12pt,preprint]{elsarticle}

\usepackage{amsmath,amsthm,amsfonts,amssymb}
\usepackage{mathtools}
\usepackage{color}

\usepackage{pifont}
\usepackage{natbib}
\usepackage{geometry}
\usepackage{fleqn}
\usepackage{graphicx}
\usepackage{txfonts}
\usepackage{hyperref}

\usepackage{cite}

\newtheorem{theorem}{Theorem}[section]
\newtheorem{lemma}[theorem]{Lemma}
\newtheorem{corollary}[theorem]{Corollary}

\theoremstyle{remark}
\newtheorem{remark}[theorem]{Remark}

\newcommand{\e}{_\varepsilon}

\newcommand{\eps}{{\varepsilon}}
\newcommand{\ds}{\displaystyle}
\renewcommand{\a}{\alpha}
\renewcommand{\b}{\beta}

\renewcommand{\d}{\hspace{2pt}\mathrm{d}}

\newcommand{\suml}{\sum\limits}
\newcommand{\intl}{\int\limits}
\newcommand{\liml}{\lim\limits}

\renewcommand{\phi}{\varphi}

\newcommand{\HS}{\mathcal{H}}           
\newcommand{\A}{\mathcal{A}}    
\renewcommand{\a}{\mathfrak{a}}    
\newcommand{\Id}{\mathrm{I}} 

\DeclareMathOperator{\dom}    {dom}
\DeclareMathOperator{\dist}   {dist}

\DeclareMathOperator{\capty}  {cap}

\newcommand{\clo}[1]{\overline{{#1}}} 
\newcommand{\abs}[2][{}]{\lvert{#2}\rvert_{{#1}}}    

\newcommand{\R}{\mathbb{R}} 
\newcommand{\N}{\mathbb{N}} 
\newcommand{\Z}{\mathbb{Z}} 

\renewcommand{\L}{\mathsf{L}^2} 
\renewcommand{\H}{\mathsf{H}^1} 

\newcommand{\restr}[1]{{\restriction}_{#1}} 
\newcommand{\Hausdorff}{\mathrm{H}}  

\newcommand{\s}{\mathrm{s}}

\allowdisplaybreaks

\begin{document}

\begin{frontmatter}

\title{ $\delta'$-interaction as a limit of a thin Neumann waveguide with transversal window}

\author[a1]{Giuseppe Cardone}
\ead{giuseppe.cardone@unisannio.it}
\address[a1]{Department of Engineering, University of Sannio, Corso Garibaldi 107, 82100 Benevento, Italy}
\cortext[cor1]{Corresponding author}

\author[a2]{Andrii Khrabustovskyi\corref{cor1}}
\ead{khrabustovskyi@math.tugraz.at}
\address[a2]{ Institute of Computational Mathematics, Graz University of Technology, Steyrergasse 30, 8010 Graz, Austria}

\begin{abstract}
We consider a waveguide-like domain consisting of two thin straight tubular domains connected through a tiny window. The perpendicular size of this waveguide is of order $\varepsilon$. Under the assumption that the window is appropriately scaled we prove that the Neumann Laplacian on this domain converges  in (a kind of) norm resolvent sense as $\varepsilon\to 0$ to a one-dimensional Schr\"odinger operator  corresponding to a $\delta'$-interaction of a non-negative strength. We estimate the rate of this convergence, also we prove the convergence of spectra.
\end{abstract}


\journal{a journal}
\end{frontmatter}

\section{Introduction\label{sec1}}

In this paper we address the problem of a geometrical approximation of
a special class of solvable models in quantum mechanics. 
Below we introduce this problem in more details.

\emph{Solvable models} describe the motion of a particle in a
potential being supported at a discrete (finite or infinite) set of
points. The term ``solvable'' reflects the fact that their mathematical and physical
quantities (spectrum, eigenfunctions, resonances, etc.) can be determined explicitly.
Note that in the literature such models are also called \emph{point interactions}.
We refer to the monograph \citep{AGHH05} for a comprehensive introduction and a detailed list of references on this topic. 

The classical example in this area is the famous Kronig-Penney model describing an electron moving
in a crystal lattice. Its mathematical representation is the one-dimensional 
Schr\"odinger operator with a singular potential supported on $\Z:=\{0,\pm1,\pm2,\dots\}$:
\begin{gather}\label{delta-formal}
-{\d^2 \over \d z^2} + \alpha\suml_{k\in\Z}\delta(\cdot -{k}),\ \alpha\in\R\cup\{\infty\},
\end{gather}
where $\delta(\cdot -{k})$ is the Dirac delta-function supported at  $k $. The formal expression \eqref{delta-formal} can be realized as a self-adjoint operator    in $\L(\R)$ 
with the action  $ -(u\restriction_{\R\setminus \Z})'' $  and the domain consisting of $\mathsf{H}^2(\R\setminus \Z)$-functions satisfying the following   conditions at   $k\in\Z$:
\begin{gather}\label{delta-conditions}
u(k-0) = u(k + 0),\quad u'(k+0)-u'( k-0)=\alpha u( k\pm 0).
\end{gather}
One says that conditions \eqref{delta-conditions} correspond to the $\delta$-interaction at the points $k $ of the strength $\alpha$. 

In the current paper we deal with another well-known model -- the so-called $\delta'$-interactions in which
the role of the values of functions and their derivatives are switched comparing with \eqref{delta-conditions}. Namely, $\delta'$-interaction at the point $z\in\R$ is given by the following coupling conditions:
\begin{gather}\label{delta'-conditions}
u'(z -0) = u'(z  + 0),\quad u(z +0)-u(z -0)=\beta u'(z\pm 0 ),\ \beta\in \R\cup\{\infty\}.
\end{gather}
One says that conditions \eqref{delta'-conditions} correspond to the $\delta'$-interaction at $z$ of the strength $\beta$. The special case $\beta=0$ leads to the one-dimensional Laplacian   in $\L(\R)$.
The case $\beta=\infty$ leads to the decoupling with the Neumann boundary conditions at $z\pm 0$.

The first rigorous treatment of $\delta'$-interactions was made in \citep{GH87}.
Further investigations were carried out in  \citep{AN03,AN06,BN13,BSW95,EKMT14,E95,GO08,GO10,KM10,KM14,M96,N03,S86} (see also \citep{BLL13} for the multidimensional version of $\delta'$-interactions).

The term \emph{$\delta'$-interaction} has the following justification.
Let $\mathcal{A}$ be the operator in $\L(\R)$ given by 
$(\A u)\restriction_{\R^\pm}=-(u\restriction_{\R^\pm})''$
and the domain consisting of $\mathsf{H}^2(\mathbb{R}\setminus\{0\})$-function satisfying \eqref{delta'-conditions} at $z=0$. It was shown in \citep{S86} that $\A$
represents a self-adjoint realization of the formal differential expression
$$-{\d^2\over \d x^2}+\beta\langle\cdot,\delta'\rangle\delta',$$
where $\langle\phi,\delta'\rangle$ denotes the action  of the distribution $\delta'$ on the test function $\phi$.
\medskip

It is well-known (see \citep[Sec.~1.3.2]{AGHH05}) that operators describing $\delta$-interactions can  be approximated in the norm resolvent topology by Schr\"odinger operators with \emph{regular} potentials having $\delta$-like profile. Therefore, sometimes instead  of the term \emph{$\delta$-interactions} one uses the name \emph{$\delta$-potentials}. 
In contrast,  $\delta'$-interactions cannot be
obtained using families of scaled zero-mean potentials\footnote{
That is, one cannot approximate  Schr\"odinger operators with $\delta'$-interactions by the regular Schr\"odinger operators of the form $-{\d^2\over \d x^2}+V\e$ with $V\e(x)={\b\over\eps^2}V(x/\eps)$, where 
$V\in\mathsf{C}_0^\infty(\R)$, $\int_{\R} V(x)\d x=0$, $\int_{\R} xV(x)\d x=-1$ (these properties imply $V\e\rightharpoonup\beta\delta'$ in $D'(\R)$ as $\eps\to 0$). }, see \citep{GM09,GH10,S86} for more details.
Nevertheless, one can approximate $\delta'$-interactions by Schr\"odinger operators with a triple of properly scaled $\delta$ potentials \citep{CS98} and even by Schr\"odinger operators with regular (but not with a $\delta'$-type profile) potentials \citep{AN00,ENZ01}.

We wish to contribute to the understanding of ways 
how point interactions can be approximated by more realistic models.
Above we have discussed approximations via Schr\"odinger operators with regular  potentials.
Another option is to use geometrical tools, namely, to treat approximations by Laplace-type 
operators on thin domains with waveguide-like geometry. 
In the current paper we address this question for  $\delta'$-interactions and Neumann Laplacians. 
Evidently, such approximations are possible only if we are given with  $\delta'$-interactions generating non-negative operators. The later holds iff  $\beta\in[0,\infty]$.

Apparently, for the first time similar problem was treated in \citep{EP05,KZ03}.
The authors of \citep{KZ03} (resp., \citep{EP05}) studied asymptotic properties of the  Neumann Laplacian (resp., the Laplace-Beltrami operator) on  narrow tubular domains (reps., tubular manifolds) consisting of two straight parts and certain ``connector''\footnote{In fact, the authors of \citep{EP05,KZ03} inspected even more general domains. In particular, they considered domains consisting of $k$ ($k\geq 2$) straight parts joined in a ``bunch'' -- in this case the limiting operator ``lives'' on a star graph with $k$ edges.}. They focused on the case, when the ``connector'' shrinks to a point  \emph{slower} comparing with a shrinking of the tubular parts. 
In the most interesting case, the obtained limiting operator corresponds to the $\delta$-coupling, but with a coupling constant depending on a spectral parameter.

Geometrical approximations of $\delta'$-interactions were considered in \citep{EP09}, but instead of Laplace-type operators the authors used Schr\"odinger operators ---  the desired coupling was generated by a suitably chosen potential (a la in \citep{CS98,ENZ01}), and not by geometrical peculiarities. 
\medskip

Below we sketch our main results.
Let $\eps>0$ be a small parameter. We consider the domains $\Omega^\pm\e\subset \mathbb{R}^{n}$ ($n\geq 2$)  of the form
$$\Omega\e^-= (L_-,0)\times \tilde S\e,\quad \Omega\e^+= (0,L_+)\times \tilde S\e,$$
where $-\infty\leq L^-<0<L^+\leq \infty$, $\tilde S\e$ is the $\eps$-homothetic image of the domain $\tilde S\subset \mathbb{R}^{n-1}$. The domains $\Omega\e^-$ and $\Omega\e^+$ are disjoint, and we connect them by means of the small ``window'' $$D\e: = \{0\}\times \tilde D\e,$$ where $\tilde D\e$ is the $d\e$-homothetic image of the domain $\tilde D\subset \mathbb{R}^{n-1}$ with $d\e\to 0$ as $\eps\to 0$; we also assume that $\tilde D\e\subset \tilde S\e$. The obtained domain 
$\Omega\e:=\Omega^-\e\cup\Omega^+\e\cup D\e$
is depicted on Figure~\ref{fig1} (for the case $n=2$).
Note, that  $\Omega\e$ is allowed to be unbounded.

\begin{figure}[h]
	\begin{center}
		\begin{picture}(350,60)
		\includegraphics[scale=0.65]{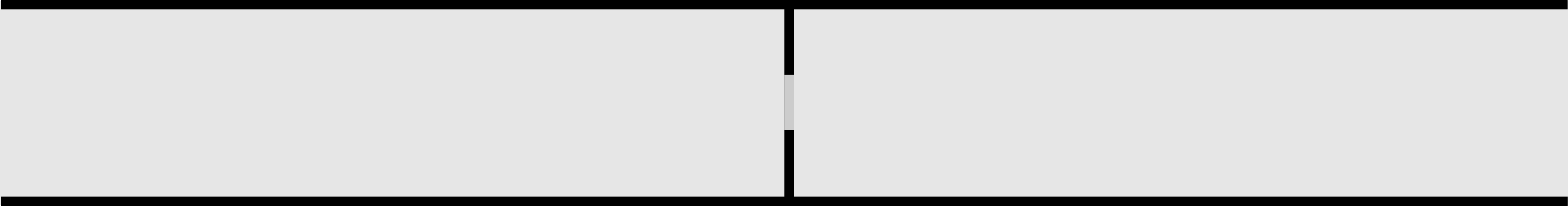}
		\put(-40,20){$\Omega\e^+$}
		\put(-315,20){$\Omega\e^-$}
		\put(-195,35){$D\e$}
		\put(-182,32){\vector(1,-1){10} }
		\put(-122,30){\vector(0,1){14} }
		\put(-122,17){\vector(0,-1){14} }
		\put(-124,23){$_{\sim\eps}$}
		
		\put(-162,23){\vector(0,1){7} }
		\put(-162,24){\vector(0,-1){7} }
		\put(-159,24){$_{\sim d\e}$}
		
		\put(10,10){\vector(1,0){15} }
		\put(10,10){\vector(0,1){15} }
		
    	\put(22,0){$^z$}
    	\put(4,16){$^{\tilde x}$}

		\end{picture}
	\end{center}
	\caption{The waveguide $\Omega\e$\label{fig1}}
\end{figure}
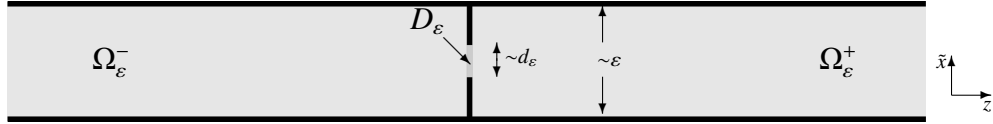

We consider the operator $\A\e=-\Delta_{\Omega\e}+V\e$, where  $-\Delta_{\Omega\e}$ is the Neumann Laplacian on $\Omega\e$ and $V\e$ is a real-valued bounded function. Our goal is to describe the behaviour of the resolvent of $\A\e$ and its spectrum as $\eps\to 0$. 
Note, that the case $V\e\equiv 0$ is not excluded; the potential $V\e$ plays no role for generating  the $\delta'$-interaction.  We treat more general operators due to the reasons explained below.

Our main results is as follows (see Theorem~\ref{th1}).
Assume, that the following limit, either finite or infinite, exists: 
$$\gamma:=\lim_{\eps\to 0}{\capty(D\e)\over   {\mu_{n-1}(\tilde S\e)}}.$$ 
Here $\capty(D\e)$ is the capacity of $D\e$ and $\mu_{n-1}(\tilde S\e)$ is the Lebesgue measure of $\tilde S\e\subset\mathbb{R}^{n-1}$. 
Also assume that $V\e$ converges in a suitable sense to the  function
$V:(L_-,L_+)\to \R$ as $\eps\to 0$ (see condition \eqref{convV} below; obviously, it holds if $V\e=V\equiv 0$).  
Then  $\A\e$ converges in (a kind of) the norm resolvent sense  to the operator  $\A^\gamma$, which acts in $\L(L_-,L_+)$ and is defined by the operation $-{\d^2 \over \d z^2}+V$ on $(L_-,L_+)\setminus\{0\}$, the Neumann conditions at $L_\pm$ (if $|L_\pm|<\infty$) and the $\delta'$-coupling at $z=0$ of the strength $4\gamma^{-1}$.
We also estimate the rate of this convergence. 

Of course we are not able to use the classical notion of the resolvent convergence since the resolvents of $\A\e$ and $\A^\gamma$ act in different Hilbert spaces (respectively, $\L(\Omega\e)$ and $\L(L_-,L_+)$). Therefore we use a suitably modified definition which involves some identification operators between those spaces.

Using the above result we then establish the Hausdorff convergence of spectra and (if $|L_\pm|<\infty$) the ``index-wise'' convergence of eigenvalues; see Theorem~\ref{th2}, Corollary~\ref{coro1} and Theorem~\ref{th3}. 
Note that if $V\e=V\equiv 0$ and $|L_\pm|<\infty$
then $\sigma(\A\e)=\sigma(\A^\gamma)=[0,\infty)$ and the convergence of spectra becomes trivial.
This is the main reason why we treat more general Schr\"odinger operators with non-zero regular potentials. 

\subsection*{Applications}
 
The obtained results can be word-by-word translated to the case of finitely many windows or to the case of a sequence of identical windows, distributed periodically along the unbounded waveguide (see Figure~\ref{fig2}). 
In the later case we arrive (as $\eps\to 0$) on the Schr\"odinger operator with a periodic sequence of $\delta'$-interactions. It is known (see \citep[Th.~3.6]{AGHH05}) that the spectrum of this operator  has  infinitely many gaps provided $\beta\not= 0$. Then, using the Hausdorff convergence of spectra, we conclude that for any $m\in\N$  the Neumann Laplacian on such periodic domains has at least $m$ gaps provided $\eps$ small enough.

\begin{figure}[h]
	\begin{center}
		\includegraphics[scale=0.4]{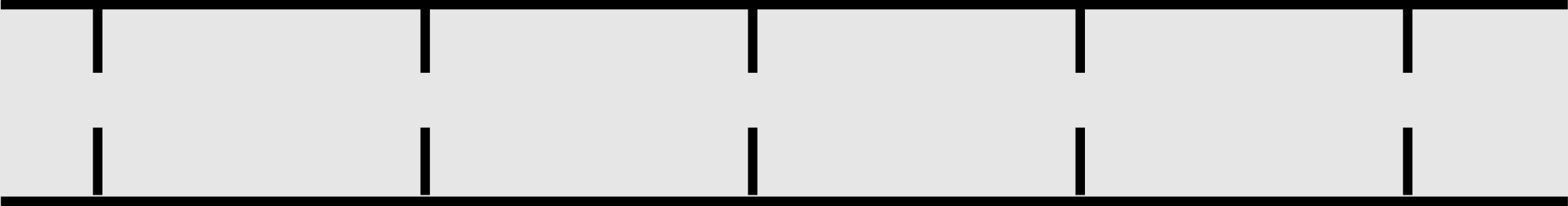}
	\end{center}
	\caption{Periodic waveguide\label{fig2}}
\end{figure}

	Note, that the idea to use periodic waveguide-like domains with period cells being  connected through small windows (or certain small passages) in order to create spectral gaps is not new -- see, e.g., \citep{BRT14,Bo15,Na10,Pa10,P03,Y03}. In these papers the transversal diameter of the waveguide is \emph{fixed} (in our paper it tends to zero as $\eps\to 0$). This leads to a complete decoupling as the windows diameters go to zero -- the limit operator is simply a direct sum of Laplacians on a sequence of identical domains. The spectrum of this decoupled operator is purely point, each eigenvalue has infinite multiplicity. Consequently, the spectrum of the Neumann Laplacian on such a domain has small bands (they shrink to those eigenvalues), while the gaps are relatively large (they tend to the gaps between those eigenvalues). 
	
	In contrast, the spectrum of the Neumann Laplacian on the domain $\Omega\e$ considered in the present paper has bands and gaps of relatively comparable size that follows easily from the spectral properties of the limiting operator $\A^\gamma$, we again refer to \citep[Chapter~III.3, Th.~3.6]{AGHH05} for the precise description of $\sigma(\A^\gamma)$.\smallskip

The paper is organized as follows. In Section~\ref{sec2} we set the problem precisely and formulate the main results.
In Section~\ref{sec3} we present some known (except Theorem~\ref{th-Haus}, which is apparently new) abstract theorems for studying the convergence of operators in varying Hilbert spaces.
Using them we prove the main results in Section~\ref{sec4}.

\section{Setting of the problem and the main result\label{sec2}}

Let $\eps\in (0,1)$  be a small parameter. 
Let $d\e$ be a positive number such that
\begin{gather}\label{de}
d\e\leq\eps.
\end{gather}

For  $n\in\mathbb{N}\setminus\{1\}$ we denote by $\tilde x=(x_1,\dots,x_{n-1})$ and $x=(\tilde x,z)$  the Cartesian coordinates in $\mathbb{R}^{n-1}$ and $\mathbb{R}^{n}$, respectively. Let $\tilde S$ and $\tilde D$ be  open bounded domains in $\mathbb{R}^{n-1}$ having Lipschitz boundaries and
satisfying
\begin{gather}\label{omegaD}
 \overline{B(\tilde D)}\subset \tilde S,
\end{gather}
where $B(\tilde D)\subset \mathbb{R}^{n-1}$ is the smallest ball containing $\tilde D$. 
For simplicity, we assume that its center coincides with the origin.

Let $$-\infty\leq L^-<0<L^+\leq \infty.$$
We denote  
\begin{gather*}
\Omega^-\e=\left\{x=(\tilde x,z)\in\mathbb{R}^n:\ \eps^{-1}\tilde x\in  \tilde S,\ L^-< z< 0\right\},
\\
\Omega^+\e=\left\{x=(\tilde x,z)\in\mathbb{R}^n:\ \eps^{-1}\tilde x\in   \tilde S,\ 0<z<L^+ \right\},
\\
D\e=\left\{x=(\tilde x,z)\in\mathbb{R}^n:\ (d\e)^{-1}\tilde x\in   \tilde D,\  z=0\right\},
\end{gather*}
In view of \eqref{de}-\eqref{omegaD},   $D\e\subset  \partial \Omega^{\pm}\e\cap \{x=(\tilde x,z)\in\mathbb{R}^n:\ z=0\}$. 
Finally, we introduce the waveguide $\Omega\e$ consisting of the cylinders $\Omega^\pm\e$  and the ``window'' $D\e$ connecting them (see Figure~\ref{fig1}):
$$\Omega\e=\Omega^-\e\cup\Omega^+\e\cup D\e.$$

Further, we will also use the notations
$$\Omega=(L^-,L^+),\quad \Omega^-=(L^-,0),\quad \Omega^+=(0,L^+),\quad \tilde S\e=\eps \tilde S.$$

We denote
\begin{gather}
\label{gamma-eps}
\gamma\e:={\capty(D\e)\over   {\mu_{n-1}(\tilde S\e)}}.
\end{gather}
Hereinafter $\mu_{d}(B)$ stands for the Lebesgue measure of $B\subset\R^d$, and by $\capty(D\e)$ we denote the capacity of $ {D}\e$ (its definition is given below). 
We assume that the  limit
\begin{gather}\label{gamma}
\gamma:=\liml_{\eps\to 0}\gamma\e,
\end{gather}
either finite or infinite, exists.

Below we recall the definition of the capacity, for more details see, e.g.,~\citep{T96}. Note, that $\capty(D\e)>0$ despite $\mu_n(D\e)=0$.

For $n\geq 3$  the capacity of the set $D\subset\R^n$ is defined via
\begin{gather}
  \label{cap-min}
  \capty(D)=\min_\psi\int_{\R^n}|\nabla \psi(x)|^2\d x,
\end{gather}
where the minimum is taken over $\psi\in \mathsf{C}_0^\infty(\R^n)$ being
equal to $1$ on a neighbourhood of $D$ (with the neighbourhood depending upon $\psi$).
For $n=2$ the right-hand-side of~\eqref{cap-min} is zero for an arbitrary domain $D$, and we use a modified definition:
\begin{gather*}
  \capty(D)=\min_\psi\int_{B_1}|\nabla \psi(x)|^2\d x,
\end{gather*}
where $B_1$ is the unit ball, which is concentric with the smallest
ball $B(D)$ containing $D$ (here we assume that $D$ is small enough
so that   $\overline{B(D)}\subset B_1$), 
the minimum is taken over $\psi\in \mathsf{C}_0^\infty(B_1)$ being
equal to $1$ on a neighbourhood of $D$.

Due to  simple rescaling arguments,
\begin{gather}\label{cap-as1}
\capty(D\e)=(d\e)^{n-2}\capty(D),\ n\geq 3,
\end{gather}
where $D:=(d\e)^{-1}D\e$.
Moreover, one can show (see, e.g.,~\citep[Lemma~3.3]{CDG02}) that
\begin{gather}\label{cap-as2}
\capty(D\e)\sim 2\pi |\ln d\e|^{-1}\text{ as }\eps\to 0,\ n=2.
\end{gather}

Let us introduce the operator $\A\e$ which will be the main object of our interest in this paper. 
Let $\{V\e\in\mathsf{L}^\infty(\Omega\e),\ \eps>0\}$ be a family of real-valued functions satisfying
\begin{gather}\label{supV} 
\mathrm{sup}_\eps\|V\e\|_{\mathsf{L}^\infty(\Omega\e)}<\infty.
\end{gather}
We denote by $\a\e$ the sesquilinear form in the Hilbert space $\L(\Omega\e)$ defined by
\begin{gather*}
\a\e[u,v]=\int_{\Omega\e}\left(\nabla u\cdot\nabla\bar{v} +V\e\, u\bar v\right) \d x
\end{gather*}
with the form domain
$\dom(\a\e)=\H(\Omega\e).$
The form $\a\e$ is densely defined, lower semibounded  and closed. 
By the first representation theorem  \citep[Chapter 6, Theorem 2.1]{K66}  there exists the unique self-adjoint  operator $\A\e$ associated with this form, i.e.
\begin{gather*}
(\A\e u,v)_{\L(\Omega\e)}= \a\e[u,v],\quad\forall u\in
\dom(\A\e),\ \forall  v\in \dom(\a\e).
\end{gather*}
If $V\e=0$ then $\A\e$ is the Neumann Laplacian on $\Omega\e$.

To simplify the presentation we restrict ourselves to the case
$$V\e(x)\geq 0$$
(and hence $\A\e$ are non-negative operators). The general case needs slight modifications.\medskip

Our first goal is to prove a kind of the norm resolvent
convergence for the operator $\A\e$.
Since $\Omega\e$ shrinks to $\Omega$ as $\eps\to 0$ we expect that the suitable limit operator (we denote it $\A^\gamma$) acts in the space $\L(\Omega)$. 
Of course, the usual notion of the norm resolvent convergence cannot be applied here since the resolvent of $\A\e$ and the  resolvent of $\A^\gamma$ ``live''
in the different Hilbert spaces spaces $\L(\Omega\e)$ and $\L(\Omega)$, respectively, and therefore we are not able to evaluate their difference. 
Therefore, the classical definition should be appropriately modified. The modified definition should be
adjusted in such a way that it still implies the convergence of spectra as it takes place in a fixed Hilbert space (see, e.g., the classical Kato's monograph \citep{K66} for more details).

The standard approach (see, e.g., \citep{IOS89,V81}) is to treat the operator $T\e:\L(\Omega)\to\L(\Omega\e)$,
\begin{gather*}
T\e:=R\e J\e - J\e R,
\end{gather*}
where $R\e=(\A\e+\Id)^{-1}$ and $R=(\A^\gamma+\Id)^{-1}$ are the resolvents of $\A\e$ and $\A^\gamma$, respectively, and 
$J\e:\L(\Omega)\to \L(\Omega\e)$ is a suitable bounded linear operator satisfying
\begin{gather*}
\forall f\in\L(\Omega):\ \liml_{\eps\to 0}\|J\e f\|_{\L(\Omega\e)} =\|f\|_{\L(\Omega)}
\end{gather*}  
(so, roughly speaking, the operator $J\e$ is  ``almost'' isometric   for small $\eps$).

In our context, the natural choice for the operator $J\e$  is  
\begin{gather}\label{J}
(J\e f)(x)=
{1\over \sqrt{\mu_{n-1}(\tilde S\e)}}f(z),\quad x=(\tilde x,z)\in\Omega\e
\end{gather}
Evidently, for each $f\in \L(\Omega)$ 
\begin{gather}\label{CS} 
\|J\e f \|_{\L(\Omega\e)}=\|f \|_{\L(\Omega)}.
\end{gather}

Alternatively, one can study the operator $\check T\e:\L(\Omega\e)\to\L(\Omega)$ defined by
\begin{gather*}
\check T\e:=\check J\e R\e - R \check J\e
\end{gather*}
with an appropriate linear bounded operator $\check J\e:\L(\Omega\e)\to \L(\Omega)$ (again ``almost'' isometric in a certain sense, cf.~\eqref{almost-is}).  In our case, the natural choice for $\check J\e$ is
\begin{gather}
\label{J'}
(\check J\e u)(z)=
{1\over\sqrt{\mu_{n-1}(\tilde S\e)}}\int_{\tilde S\e} u(\tilde x,z)\d \tilde x,\quad z\in\Omega.
\end{gather}
Evidently, $(J\e)^*=\check J\e$, moreover $\check J\e$ is the left inverse operator for $J\e$, i.e. $\check J\e J\e =\Id$. 
Further, using the Cauchy-Schwarz and the Poincar\'e inequalities for $\tilde S\e $, one gets  
\begin{gather}\label{CS-P1}
\forall u\in\L(\Omega\e):\ \|\check J\e u\|_{\L(\Omega)}\leq \|u\|_{\L(\Omega\e)},\\[2mm]\label{CS-P2}
\forall u\in\H(\Omega\e):\ \|u\|^2_{\L(\Omega\e)}\leq \|\check J\e u\|^2_{\L(\Omega)} + \eps^2 \lambda^{-1}_{\tilde S}\|\nabla u\|^2_{\L(\Omega\e)},
\end{gather}
the constant $\lambda_{\tilde S}$ is the smallest non-zero eigenvalue of the Neumann Laplacian on $\tilde S$.  
Hence
\begin{gather}
\label{almost-is}
\|u\e\|^2_{\L(\Omega\e)}= \|\check J\e u\e\|^2_{\L(\Omega)}+o(1)\text{ as }\eps\to 0
\end{gather}
for any family $(u\e)_{\eps>0}$ with $u\e\in \H(\Omega\e)$ satisfying $\sup_{\eps>0}\|\nabla u\e\| _{\L(\Omega\e)}<\infty$.

\medskip

Below we introduce the expected limiting operator $\A^\gamma$.
It is convenient to define it via a form approach. 

We start from the case $\gamma<\infty$. Let $ \a^\gamma $ be a sesquilinear form in $\L(\Omega)$ 
defined by
\begin{multline*}
\a^\gamma[f,g]=\int_{\Omega^+} f'(z)\overline{g'(z)} \d z+\int_{\Omega^-} f'(z)\overline{g'(z)} \d z+\int_\Omega V(z)f(z)\overline{g(z)}\d z\\ +\,
{\gamma\over 4}\big(f(+0)-f(-0)\big)\overline{\big(g(+0)-g(-0)\big)},\quad 
\dom(\a^\gamma)=\H(\Omega\setminus\{0\}).
\end{multline*}
where $V\in\mathsf{L}^\infty(\Omega)$ is a real-valued function such that $V(x)\ge 0$, by $f(\pm 0)$ and $g(\pm 0)$ we denote the limiting values of $f(z)$ and $g(z)$ as $z\to \pm 0 $.
This form is non-negative and closed.
We denote by $\A^\gamma$ the self-adjoint  operator
  being associated with $\a^\gamma $.
It is easy to show that 
\begin{gather*}
\dom(\A^\gamma)=
\left\{
f\in \mathsf{H}^2(\Omega\setminus\{0\}): 
\begin{array}{l}
f'(-0)=f'(+0),\ f(+0)-f(-0)=4\gamma^{-1}f'(\pm 0),\\[1mm]
f'(L^-)=0\text{ provided }L^->-\infty,\
f'(L^+)=0\text{ provided }L^+<\infty
\end{array}
\right\},
\\[1mm] 
(\A^\gamma u)(z)=
\begin{cases}
-f''(z)+V(z),&z<0,\\
-f''(z)+V(z),&z>0,\\
\end{cases}
\end{gather*}
i.e. $\A^\gamma $ is the {Schr\"odinger operator with  the $\delta'$-interaction at $0$ of the strength $4\gamma^{-1}$}. 

In the case $\gamma=\infty$ we define $ {\A}^\infty$ as the operator acting in $\L(\Omega)$ and 
being associated with the form $$\a^\infty[u,v]=\int_{\Omega} f'(z)\overline{g'(z)}\d z + \int_{\Omega}V(z)f(z)\overline{g(z)} \d z,\quad\dom(\a^\infty)=\H(\Omega).$$ It is clear that
\begin{gather*}
\dom( {\A^\infty})=\left\{f\in \mathsf{H}^2(\Omega):\ f'(L^-)=0\text{ as }L^->-\infty,
f'(L^+)=0\text{ as }L^+<\infty\right\},\\ \A^\infty f=-f''+ V.
\end{gather*}

Of course the potentials $V\e$ and $V$ have to be close in a suitable sense in order to guarantee 
the closeness of the resolvents of the underlying operators.
Namely, as it follows from Theorem~\ref{th1} below, one needs
\begin{gather}
\label{convV}
\|V\e - \sqrt{\mu_{n-1}(\tilde S\e)}J\e V \|_{\mathsf{L}^\infty(\Omega\e)}\to 0\text{ as }\eps\to 0.
\end{gather}

We are now in position to formulate the mains results of this work. Below $\|\cdot\|_{X\to Y}$ stands for the operator norm of an operator acting between normed spaces $X$ and $Y$.
By $C$, $C_1$ etc.\ we denote generic constants being independent of $\eps$ and of functions 
standing at the estimates  where these constants occur (but they may depend on $n$, $\tilde S$, $\tilde D$, $L_+$, $L_-$, $\gamma$, $\sup\gamma\e$ or $\mathrm{sup}_\eps\|V\e\|_{\mathsf{L}^\infty(\Omega\e)}$).

\begin{theorem}\label{th1}
One has
\begin{gather*}
\left\|(\A\e+ \Id)^{-1} J\e- J\e (\A^\gamma+ \Id)^{-1}  \right\|_{\L(\Omega)\to \L(\Omega\e)}\leq  4\delta^\gamma\e,
\\[1mm]
\left\|\check J\e(\A\e+ \Id)^{-1} -  (\A^\gamma+ \Id)^{-1} \check J\e \right\|_{\L(\Omega\e)\to \L(\Omega)}\leq  6\delta^\gamma\e,
\end{gather*}
where 
\begin{gather}\label{delta}
\delta^\gamma\e=
 \|V\e - \sqrt{\mu_{n-1}(\tilde S\e)}J\e V \|_{\mathsf{L}^\infty(\Omega\e)}+
C
\left\{
\begin{array}{lll}
\eps^{1/2}|\ln\eps|,&\gamma<\infty,&n=2,\\[1mm]
\eps^{1/2},&\gamma<\infty,&n\ge 3,\\[1mm]
\eps^{1/2}+  \gamma\e ^{-1/2},&\gamma=\infty.&
\end{array}
\right.
\end{gather}
\end{theorem}

As an important application of this norm resolvent convergence type result, we establish the 
Hausdorff convergence of spectra. 
For two compact sets $X,Y\subset\R$  we define the \emph{maximal outside distance} and \emph{maximal inside distance}  of $X$ to $Y$ by
$$\dist_{\rm out}(X,Y):=\sup_{x\in X} \dist(x,Y),\quad \dist_{\rm in}(X,Y):=\dist_{\rm out}(Y,X),$$
where $\dist(x,Y)=\inf_{y\in Y}|x-y|$.
Finally, we define the \emph{Hausdorff distance} between   $X$ and $Y$:
$$\dist_\Hausdorff(X,Y):=
\max\left\{\dist_{\rm out}(X,Y),\ \dist_{\rm in}(X,Y)\right\}.$$ 
Note, that
\begin{itemize} 
\item
$\dist_{\rm out}\left(X\e,X\right)\to 0\text{ as }\eps\to 0$ iff
for each $x \in \mathbb{R}\setminus X$ there exists $d>0$ such that $X\e\cap\{x\e:\ |x\e-x|<d\}=\emptyset$. 

\item  $\dist_{\rm in}\left(X\e,X\right)\to 0\text{ as }\eps\to 0$ iff
for each $x\in X$ there exists a family  $\{x\e\}\e$ with $x_\eps \in X_\eps$ such that
  $\lim_{\eps\to 0}x\e=x$,
\end{itemize}
where  $\{X\e\}_{\eps}$ is a family of compact sets in $\R$, $X\subset\R$ is also a compact set.

Since the spectra of $\A\e $ and $\A^\gamma$ are noncompact sets, it is reasonable to measure the Hausdorff distance between  the spectra of their resolvents
$( \A\e  +\Id)^{-1}$ and $ (\A^{ \gamma} +\Id)^{-1}$.
Note, that by spectral mapping theorem
$\sigma((\A\e  +\Id)^{-1})=\overline{(\sigma(\A\e )+1)^{-1}},$ $\sigma((\A^{ \gamma} +\Id)^{-1})=\overline{(\sigma(\A^{ \gamma})+1)^{-1}}.$

Recall, that $\lambda_{\tilde S}$ is the smallest non-zero eigenvalue of the Neumann Laplacian on $\tilde S$ and $\delta\e^\gamma$ is defined by \eqref{delta}.

\begin{theorem}\label{th2}
For each $d\in (0,1)$ one has
$$\dist_{\rm out}\left(\sigma((\A\e  +\Id)^{-1})\cap[d,1],\,\sigma((\A^\gamma   +\Id)^{-1})\right)\leq 
{6 \delta\e^\gamma\over  \sqrt{1-(\lambda_{\tilde S})^{-1}\eps^2(d^{-1}-1)}}$$
provided $\eps<\left({\lambda_{\tilde S} d\over 1-d}\right)^{1/2}$.
Moreover,
$$\dist_{\rm in}\left(\sigma((\A\e  +\Id)^{-1}),\sigma((\A^\gamma  +\Id)^{-1}))\right)\leq  4\delta\e^\gamma.$$ 
\end{theorem}

Taking into account that $0\in \sigma((\A\e  +\Id)^{-1})\cap\sigma((\A^\gamma  +\Id)^{-1})$,
we obtain easily the following corollary from Theorem~\ref{th2}.

\begin{corollary}\label{coro1}
Let \eqref{convV} hold. 
Then
\begin{gather*}
\dist_\Hausdorff 
  \left(\sigma((\A\e +\Id)^{-1}),\,\sigma((\A^{ \gamma}+\Id)^{-1})\right)
  \to 0\ \text{ as }\eps\to 0.
\end{gather*}
\end{corollary}

In the subsequent theorem we assume that $L_->-\infty,\ L_+<\infty$. In this case the spectra of $\A \e$ and $\A^{ \gamma}$ are purely discrete.
We denote by $\{\lambda _{k,\eps}\}_{k\in\mathbb{N}}$ and
$\{\lambda_{k}^{ \gamma}\}_{k\in\mathbb{N}}$ the sequences of the eigenvalues of $\A\e $ and $\A^{ \gamma}$, respectively,
arranged in the ascending order and repeated according to their
multiplicities.
 
\begin{theorem}\label{th3}
Let \eqref{convV} hold. Then  
\begin{gather}\label{spectrum1}
\lim_{\eps\to 0}\lambda_{k,\eps} =\lambda_k^{ \gamma},\ k\in\mathbb{N},
\end{gather}
moreover
\begin{gather}
\label{spectrum2}
\left|(\lambda_{k,\eps} +1)^{-1}-(\lambda_k^{ \gamma}+1)^{-1}\right|\leq
   4C\e \delta^{ \gamma}\e,
\end{gather} 
where the constants $C\e$ satisfies  $|C\e|\leq C$, $\lim_{\eps\to 0}C\e=1$.
\end{theorem} 
 
In the next section we present  abstract results, which are then being used for the proofs of Theorems~\ref{th1}-\ref{th3}. Some of them are known  \citep{IOS89,P06}, while Theorem~\ref{th-Haus} is apparently new and is interesting by itself. 

\section{Abstract framework}\label{sec3}

\subsection{Operator estimates for resolvent difference} 

Let $\HS$ and $ \HS\e$ be two separable Hilbert spaces. 
Let $\a$ and $\a\e$ be closed  densely defined  non-negative
sesquilinear forms in $\HS$ and $\HS\e$, being associated with  the non-negative  self-adjoint operators
$\A$ and $\A\e$, respectively. Note, that within this subsection $\HS\e$ is just a \textit{notation} for some Hilbert space, which (in general) differs from the space $\HS$, i.e.\ the sub-index $\eps$ does not mean that this space depends on a small parameter. 
  
We also introduce Hilbert spaces $\HS^1$, $\HS^2$ and $\HS^1\e$, $\HS^2\e$  via
\begin{gather}\label{scale}
\begin{array}{lll}
\HS^1=\dom(\a)&\text{equipped with the norm }&\|f\|_{\HS^1}=(\a[f,f]+\|f\|_{\HS}^2)^{1/2},\\[1mm]
\HS^1\e=\dom(\a\e)&\text{equipped with the norm }&\|f\|_{\HS^1\e}=(\a\e[f,f]+\|f\|_{\HS\e}^2)^{1/2},\\[1mm]
\HS^2=\dom(\A)&\text{equipped with the norm }&\|f\|_{\HS^2}=\|\A f+f\|_{\HS},\\[1mm]
\HS^2\e=\dom(\A\e)&\text{equipped with the norm }&\|f\|_{\HS^2\e}=\|\A\e f+f\|_{\HS\e}.
\end{array}
\end{gather}
Evidently, $\HS^2\subset\HS^1\subset\HS$, moreover
\begin{gather}\label{scale+}
\forall f\in\HS^2:\ \|f\|_{\HS}\leq \|f\|_{\HS^1}\leq \|f\|_{\HS^2};
\end{gather}
similar statement holds true for $\HS\e$, $\HS^1\e$, $\HS^2\e$.

Finally, let $$J\e\colon
  \HS\to {\HS\e},\ {\check J\e}\colon {\HS\e}\to \HS,\ {J\e^1} \colon {\HS^1} \to {\HS\e^1},\ {\check J\e^{1}} \colon {\HS\e^1}\to {\HS^1}$$
  be linear bounded
  operators. In the applications  the operators $J\e$ and $\check J\e$ appear in a natural way (as, for example, $J\e$ \eqref{J} and $\check J\e$ \eqref{J'} in our case), while the other two operators
should be constructed as 
``almost'' restrictions of $J\e$ and $\check J\e$ to $\HS^1$ and $\HS^1\e$, respectively (see conditions \eqref{cond1} below).

\begin{theorem}[\citep{P06}]
\label{th0} 
Let for some $\delta\e\geq 0$ and $k\leq 2$ the following conditions hold:
\begin{subequations}\label{cond123}
\begin{align}
\begin{array}{ll}
\|J\e f-J\e^1 f\|_{\HS\e}\leq \delta\e\|f\|_{\HS^1},&  \forall f\in \HS^1,\\[1mm]    
\|\check  J\e u-\check J\e^{1} u\|_{\HS}\leq 
\delta\e\|u\|_{\HS\e^1},& \forall u\in \HS\e^1,
\end{array} \label{cond1}
\\[1mm]
\left|(J\e f,u)_{\HS\e}-(f, \check J\e u)_{\HS}\right|\leq 
\delta\e \|f\|_{\HS}\|u\|_{\HS\e},&\quad \forall f\in \HS,\ u\in \HS\e,\label{cond2}\\[1mm]
\left|\a\e(J^1\e f,u)-\a(f, \check J\e^{1} u)\right|\leq 
\delta\e \|f\|_{\HS^k}\|u\|_{\HS\e^1},&\quad \forall f\in \HS^k,\ u\in \HS\e^1.\label{cond3}
\end{align}
\end{subequations} 
Then 
\begin{gather}\label{nrc4}
\left\|(\A\e+\Id)^{-1}J\e -J\e(\A+\Id)^{-1} \right\|_{\HS\to \HS\e}\leq 4\delta\e,\\ \label{nrc6}
\left\|\check J\e(\A\e+\Id)^{-1} - (\A+\Id)^{-1} \check J\e \right\|_{\HS\to \HS\e}\leq 6\delta\e.
\end{gather}

\end{theorem} 
  
\begin{remark} 
Estimate \eqref{nrc4} is the first statement of \citep[Th.~A.5]{P06},
estimate \eqref{nrc6} is the particular case of the first statement of \citep[Th.~A.10]{P06}.
We note that on the first glance \citep[Th.~A.10]{P06} requires more assumptions on $J\e$ and $\check J\e$, namely 
$$
\|f-\check J\e J\e f\|_{\HS  \to\HS  }\leq\delta\e\|f\|_{\HS^1  },\quad
\|u-J\e\check  J\e u\|_{\HS\e\to\HS\e}\leq\delta\e\|u\|_{\HS\e^1}.
$$
But it is visible from its proof that these additional assumptions are needed only for the second and the third statements of \citep[Th.~A.10]{P06}, while the first one needs only conditions \eqref{cond123}.
\end{remark}

\begin{remark}  
The fact that the convergence of sesquilinear forms with \emph{common domain} implies norm
resolvent convergence of the associated operators is well known, see e.g. \citep[Th.~VIII.25]{RS72}.
Theorem~\ref{th0} is a version of this statement for the case of varying Hilbert spaces.
\end{remark}

\subsection{Hausdorff distance between spectra}  

As in the previous subsection, $\HS$ and $ \HS\e$ are two separable Hilbert spaces, 
$\a$ and $\a\e$ are closed  densely defined  non-negative
sesquilinear forms in $\HS$ and $\HS\e$, and 
$\A$ and $\A\e$ are the operators associated with these forms.

\begin{theorem}\label{th-Haus}
Let 
\begin{gather}\label{th-Haus-1}
\left\|(\A\e+\Id)^{-1}J\e -J\e(\A+\Id)^{-1} \right\|_{\HS\to \HS\e}\leq \delta\e,\quad
\left\|\check J\e(\A\e+\Id)^{-1} - (\A+\Id)^{-1} \check J\e \right\|_{\HS\e\to \HS}\leq \check\delta\e.
\end{gather}
Here  $\delta\e$ and $\check\delta\e$ are non-negative constants, $J\e\colon
\HS\to {\HS\e},\ {\check J\e}\colon {\HS\e}\to \HS$ are linear bounded operators satisfying 
\begin{gather}\label{Jcond3}
\|f\|^2_{\HS}\leq \|J\e f\|^2_{\HS\e}+\nu\e \a[f,f],\quad\forall f\in \dom(\a), \\\label{Jcond4}
\|u\|^2_{\HS\e}\leq \|\check J\e u\|^2_{\HS}+\check\nu\e \a\e[u,u],\quad\forall u\in \dom(\a\e),
\end{gather}
where  $\nu\e,\check{\nu}\e $ are non-negative constants.
 
Then for any $d\in (0,1)$ one gets
\begin{gather*}
\dist_{\rm out}\left(\sigma((\A\e  +\Id)^{-1})\cap[d,1],\,\sigma((\A   +\Id)^{-1})\right)\leq 
{ \check\delta\e \over \sqrt{1-\check\nu\e(d^{-1}-1)}},\\ 
\dist_{\rm out}\left(\sigma((\A  +\Id)^{-1})\cap[d,1],\,\sigma((\A\e    +\Id)^{-1})\right)\leq 
{ \delta\e \over \sqrt{1-\nu\e(d^{-1}-1)}} 
\end{gather*}
provided $\nu\e,\check\nu\e\in [0,{d\over 1-d})$.

\end{theorem}

\begin{remark}
For our operators $J\e$ \eqref{J} and $\check J\e$ \eqref{J'} conditions \eqref{Jcond3}-\eqref{Jcond4} hold with $ \nu\e=0$, $\hat\nu\e=\eps^2\lambda^{-1}_{\tilde S}$ (cf.~\eqref{CS}, \eqref{CS-P2}).
\end{remark}

\begin{proof}We denote $\mathcal{R}\e:=(\A\e+\Id)^{-1}$, $\mathcal{R}:=(\A+\Id)^{-1}$.

Let $z\not\in \sigma(\mathcal{R})$. 
From
$\|(\mathcal{R}-z\Id)^{-1}\|_{\HS}={1\over \dist(z,\,\sigma(\mathcal{R}))}$
 we conclude that
\begin{gather}\label{Haus1}
\forall\phi\in\HS\text{ with }\|\phi\|_{\HS}=1:\quad \dist(z,\sigma(\mathcal{R}))\leq \|(\mathcal{R}-z\Id)\phi\|_{\HS}.
\end{gather} 
For $z\in \sigma(\mathcal{R})$ the above estimate is obvious, thus it holds for all $z\in\mathbb{C}$.

Now, let $z\in\sigma(\mathcal{R}\e)\cap[d,1]$. We set $\lambda:={1-z\over z}$. By spectral mapping theorem $\lambda\in \sigma(\A\e)\cap[0,{1-d\over d}]$ and hence for each $\eta>0$ there exists $\psi_\eta\in\dom(\A\e)$ such that
\begin{gather}\label{weyl}
\|\psi_\eta\|_{\HS\e}  =1,\quad \|(\A\e-\lambda\Id)\psi_\eta\|_{\HS\e}\leq \eta. 
\end{gather} 
Using the identity $$(\mathcal{R}\e-z\Id)\psi =-z \mathcal{R}\e(\A\e-\lambda\Id )\psi,\ \psi\in\dom(\A\e)$$ and taking into account that $z\in (0,1)$, $\|\mathcal{R}\e\|_{\HS\e\to\HS\e}\leq 1$ we obtain from \eqref{weyl}:
\begin{gather}\label{Haus2}
\|(\mathcal{R}\e-z\Id)\psi_\eta\|_{\HS\e}\leq   \eta.
\end{gather}
Also we notice, that, due to \eqref{Jcond4}, \eqref{weyl},  $\check J\e\psi_\eta\not=0$ for small enough $\eta$. Indeed, taking into account that $\lambda\leq {1-d\over d}$, we obtain
\begin{multline}\label{JJ}
\|\check J\e\psi_\eta\|_{\HS}^2\overset{\eqref{Jcond4}}{\geq} \|   \psi_\eta\|_{\HS\e}^2-\check\nu\e\a\e[\psi_\eta,\psi_\eta]\\=
1-\lambda \check\nu\e -\check\nu\e(\A\e\psi_\eta-\lambda\psi_\eta,\psi_\eta)_{\HS\e}\overset{\eqref{weyl}}{\geq}
1-\left({1-d\over d}+\eta\right) \check\nu\e.
\end{multline}
Since $\check\nu\e\in [0,{d\over 1-d})$, the right-hand-side of \eqref{JJ} is positive for small $\eta$ .

Finally, 
setting $\phi:= \|\check J\e\psi_\eta\|^{-1}_{\HS} \check J\e\psi_\eta$ in \eqref{Haus1} (here we assume that $\eta$ is small enough so that $\check J\e\psi_\eta\not=0$) and taking into account \eqref{th-Haus-1}, \eqref{Haus2}, \eqref{JJ} we obtain for $z\in\sigma(\mathcal{R}\e)\cap[d,1]$: 
\begin{gather*}
\dist(z,\,\sigma(\mathcal{R}))\leq {\|(\mathcal{R}-z\Id)\check J\e\psi_\eta\|_{\HS} \over \|\check J\e\psi_\eta\|_{\HS}}\leq { \|(\mathcal{R}\check J\e-\check J\e \mathcal{R}\e)\psi_\eta\|_{\HS}+
\|\check J\e(\mathcal{R}\e - z\Id)\psi_\eta\|_{\HS}\over \|\check J\e\psi_\eta\|_{\HS}} \leq
{\check\delta\e +\|\check J\e\| \cdot\eta 
\over \sqrt{1-\check\nu\e((d^{-1}-1)+\eta)}}.
\end{gather*}
Passing to the limit $\eta\to 0$   we arrive at the estimate
\begin{gather}\label{Haus3}
\dist(z,\sigma(\mathcal{R}))\leq  
{\check\delta\e 
\over \sqrt{1-\check\nu\e(d^{-1}-1)} },\quad
\forall z\in \sigma(\mathcal{R}\e)\cap[d,1].
\end{gather}

Similarly,  we get 
\begin{gather}\label{Haus4}
\dist(z,\sigma(\mathcal{R}\e))\leq  
{\delta\e 
\over \sqrt{1- \nu\e(d^{-1}-1)} },\quad
\forall z\in \sigma(\mathcal{R})\cap[d,1].
\end{gather} 
The statement of the theorem follows directly from \eqref{Haus3}-\eqref{Haus4} and the definition of the maximal outside and maximal inside distances.
\end{proof}

\subsection{Estimate for the  difference between eigenvalues} 

For the proof of Theorem~\ref{th3} we use the abstract result
from \citep{IOS89} providing the estimate for the difference between eigenvalues of compact self-adjoint operators in varying Hilbert spaces.   

Let $\HS\e$ and
$\HS$ be separable Hilbert spaces, and
$B\e \colon \HS\e\to \HS\e,\
B \colon \HS \to \HS$ be linear
compact self-adjoint non-negative operators.
We denote by
 $\{\nu_{k,\eps}\}_{k\in\N}$ and
$\left\{\nu_k\right\}_{k\in\N}$ the eigenvalues of the
operators $B\e$ and $B$, respectively,
being
renumbered in the descending order and with account of their
multiplicity. 

\begin{theorem}[\citep{IOS89}]
  \label{thm:IOS}
  Assume that the following conditions $A_1-A_4$ hold:

{$A_1.$} The linear bounded operator $J\e \colon \HS\to
\HS\e$ exists such that for each $f\in \HS$
\begin{gather*}
\|J\e
f\|_{\HS\e} \to
\|f\|_{\HS}\text{ as }\eps\to 0.
\end{gather*}

{$A_2.$} The operator norms
$\|B\e\| $ are bounded
uniformly in $\eps$.

{$A_3.$} For any $f\in\HS$: $\|B\e J\e
f-J\e B f\|_{\HS\e}\to 0 \text{ as }\eps\to 0$.

{$A_4.$} For any family $\{f\e\in \HS\e\}_\eps$ with $\sup_{\eps}
\|f\e\|_{\HS\e}<\infty$ there exist a sequence $(\eps_m)_m$ and $w\in
\HS$ such that $ \|B_{\eps_m} f_{\eps_m}-J_{\eps_m}
w\|_{\HS_{\eps_m}}\to 0$ and $ \eps_m\to 0$ as $m \to \infty$.

Then for any $k\in\mathbb{N}$ we have
\begin{equation*}
  |\nu_{k,\eps}-\nu_k|\leq 
  C\e \sup\limits_f\|B\e J\e f-J\e B f\|_{\HS\e},
\end{equation*}
where $|C\e|\leq C$, $\lim_{\eps\to 0}C\e=1$, the supremum is taken over all
$f\in\HS$ belonging to the eigenspace associated with $\nu_k$ and
satisfying $\|f\|_\HS=1$.
\end{theorem}

In the proof of Theorem~\ref{th3} we will apply the above result for $B\e=(\A\e+\Id)^{-1}$ and $B=(\A+\Id)^{-1}$.

\section{Proofs of the main results}\label{sec4}

We denote
$$\mathcal{H}\e:=\L(\Omega\e),\quad \mathcal{H}:=\L(\Omega).$$
Also we introduce Hilbert spaces 
$$
\HS^1\e:=\dom(\a\e)=\H(\Omega\e),\quad \HS^1:=\dom(\a^\gamma)$$
with the norms being given by \eqref{scale}, namely
\begin{gather*}
\|u\|_{\HS\e^1}^2=\|\nabla u\|^2_{\L(\Omega\e)}+\|V\e^{1/2}u\|^2_{\L(\Omega\e)}+\|u\|^2_{\L(\Omega\e)},\\[1mm] \|f\|^2_{\HS^1}=
\begin{cases}
 \| f'\|^2_{\L(\Omega_-)}+\|f'\|^2_{\L(\Omega_+)}+{\gamma\over 4} \left|f(+0)-f(-0)\right|^2+\|V^{1/2}f\|_{\L(\Omega)}^2+\|f\|_{\L(\Omega)}^2 ,&\gamma<\infty,\\[1mm]
\|f'\|^2_{\L(\Omega)}+\|f\|^2_{\L(\Omega)},&\gamma=\infty.
\end{cases}
\end{gather*}
Due to a standard trace inequality $|f(\pm 0)|\leq C\|f\|_{\H(\Omega^\pm)}$, the norm $\|\cdot\|_{\HS^1}$ is equivalent to the Sobolev $\H$-norm:
\begin{gather}\label{equiv-norms}
\|f\|_{\H(\Omega\setminus\{0\})}\leq \|f\|_{\HS^1}\leq C^\gamma \|f\|_{\H(\Omega\setminus\{0\})}.
\end{gather}

Our goal is to show that conditions \eqref{cond123} hold
with $\delta\e=\delta\e^\gamma$,
$J\e\colon
  \HS\to {\HS\e}$ \eqref{J}, ${\check J\e}\colon {\HS\e}\to \HS$ \eqref{J'} and suitable  ${J\e^1} \colon {\HS^1} \to {\HS\e^1}$, ${\check J\e^{1}} \colon {\HS\e^1}\to {\HS^1}$. 
In  Subsection~\ref{subsec41} (resp., Subsection~\ref{subsec42}) we construct these operators for the case $\gamma<\infty$ (resp., $\gamma=\infty$) and prove that they enjoy the required properties.  
Then  Theorems~\ref{th1}  will follow immediately from Theorem~\ref{th0},  Theorem~\ref{th2} will follow from~\eqref{CS}, \eqref{CS-P2} and Theorems~\ref{th1},\ref{th-Haus}. The proof of Theorem~\ref{th3} needs
an additional step and we postpone it to Subsection~\ref{subsec43}.

\subsection{The case $\gamma<\infty$\label{subsec41}} 
 
In what follows we use the notations
$$
\begin{array}{lll}
Y\e:=\left\{x=(\tilde  x,z)\in \Omega\e:\, z\in (-\eps,\eps)\right\},&\ Y\e^-:=\Omega^-\e\cap Y\e,&\ Y\e^+:=\Omega^+\e\cap Y\e.
\end{array}
$$

To construct an appropriate operator $J\e^1:\mathcal{H}^1\to \mathcal{H}\e^1$ we need some preparations.
Let $\varphi:\mathbb{R}\to \mathbb{R}$ be a twice-continuously
differentiable function such that
\begin{gather*}0\leq \varphi(t) \leq 1,\\
\varphi(t)=1\text{ as }t\leq 1/2\text{ and
}\varphi(t)=0\text{ as }t\geq 1.
\end{gather*}
For $x\in \R^n$ we set
$$\phi\e(x):=\phi\left(2{|\tilde x|-\eps\kappa_1\over \eps\kappa_2}\right)\cdot\phi\left({|z|\over\eps}\right),
\quad x=(\tilde x,z)\in\R^n,$$
where $\kappa_1$ is the diameter of $B(\tilde D)$ (recall, that this notation stands for the smallest ball containing $\tilde D$; this ball has its center at  the origin),
$\kappa_2$ is the distance from $B(\tilde D)$  to  the boundary of $\tilde S$; due to \eqref{omegaD}, $\kappa_2>0$.
The function $\phi\e$ is supported on $Y\e$ and it is equal to $1$ in a neighbourhood of $D\e$.

Let $\psi\e(x)$ be the unique solution to the problem 
\begin{gather*} 
  \begin{cases}
    \Delta \psi\e(x)=0,&x\in\R^n\setminus\clo{D\e},\\
    \psi\e(x)=1,&x\in\partial  {D}\e,\\
    \psi\e(x)\to 0,&|x|\to\infty
  \end{cases}
\end{gather*}
as $n\ge 3$, or to the problem
\begin{gather}\label{capBVP2} 
  \begin{cases}
    \Delta \psi\e(x)=0,&x\in B_1\setminus\overline{D\e},\\
    \psi\e(x)=1,&x\in \partial D\e,\\
    \psi\e(x)=0,&x\in\partial B_1.
  \end{cases}
\end{gather}
as $n=2$ (recall, that $B_1$ is the unit ball  concentric with the smallest ball containing $D\e$); in the later case we extend it by zero to the whole of $\R^2$. It is known, 
that
$$
\mathrm{cap}(D\e)=\|\nabla\psi\e\|^2_{\L(\R^n\setminus\overline{D\e})}.$$

Due to a standard regularity theory,
$\psi\e$ belongs to $
\mathsf{C}^\infty(\mathbb{R}^n\setminus\overline{D\e})$ as $n\geq 3$ and to $ 
\mathsf{C}^\infty(B_1\setminus\overline{D\e})$ as $n=2$. Moreover, using  symmetry arguments, one concludes that
$\psi\e(\tilde x,z)=\psi\e(\tilde x,-z)$. Consequently,
\begin{gather}\label{psi-n1}
{\partial\psi\e\over\partial z}=0\text{ on } S\e\setminus \overline{D\e},\\
\label{psi-n2}
\left.{\partial\psi\e\over\partial z}\right|_{z=+0}=-\left.{\partial\psi\e\over\partial z}\right|_{z=-0}\text{ on } D\e,
\end{gather}
where $S\e:=\tilde S\e\times\{0\}= \left\{x=(\tilde x,z)\in\mathbb{R}^n:\ \eps^{-1}\tilde x\in   S,\  z=0\right\}$. \smallskip

Further, we will need some known pointwise estimates for the functions $\psi\e(x)$ at some positive distance from $D\e$.  

\begin{lemma}{\rm(\citep[Lemma~2.4]{MK06})}
\label{lemma-Hest}
Let  $B(D\e)$ be the smallest ball containing $D\e$.
We denote by $\rho(x)$ the
distance from $x$ to $B(D\e)$. Let $x\in\R^n\setminus\overline{B(D\e)}$ with $\rho(x)\geq C_0 d\e$ as $n\geq 3$ and  $\rho(x)\geq \exp(-C_0\sqrt{|\ln d\e|})$ as $n=2$, where $C_0$ is some positive constant. 
Then
\begin{gather*}
\begin{array}{lll}
\abs{\psi\e(x)}\leq 
C\dfrac{ (d\e)^{n-2}}{(\rho(x))^{n-2}},&
\abs{\nabla \psi\e(x)}\leq 
C\dfrac{ (d\e)^{n-2}}{(\rho(x))^{n-1}}&\text{as }n\geq 3,\\[2ex]
\abs{\psi\e(x)}\leq 
C\dfrac{ |\ln d\e|^{-1}}{|\ln  \rho(x) |^{-1} },&
\abs{\nabla \psi\e(x)}\leq 
C\dfrac{ |\ln d\e|^{-1} }{  \rho(x) }&\text{as }n=2.
\end{array}
\end{gather*}
\end{lemma} 

We will  need also another simple lemma providing us with a formula for the capacity of $D\e$.
\begin{lemma}\label{another-cap}One has
\begin{gather*} 
\capty(D\e)=2\int_{D\e}\left.{\partial \psi\e\over \partial z}\right|_{z=-0}\d\s=-2\int_{D\e}\left.{\partial \psi\e\over \partial z}\right|_{z=+0}\d\s,
\end{gather*}
where $\d \s$ is the area measure on $\partial D\e$.
\end{lemma}

\begin{proof}
Let $n=2$. Due to \eqref{capBVP2}, one has the following Green's formula:
\begin{gather}\label{green}
\intl_{B_1\setminus \overline{D\e}}|\nabla \psi\e|^2\d x= 
\intl_{D\e}\left(\left.{\partial \psi\e\over \partial z}\right|_{z=-0} - \left.{\partial \psi\e\over \partial z}\right|_{z=+0}\right).
\end{gather}
Combining \eqref{green} 
and \eqref{psi-n2} we obtain the desired result. 

For the case $n\geq 3$ we again use \eqref{psi-n2} and \eqref{green} (now with $\R^n\setminus \overline{D\e}$  instead of $B_1\setminus \overline{D\e} $). The fulfilment of \eqref{green} in the unbounded domain $\R^n\setminus \overline{D\e}$ is guaranteed by  Lemma~\ref{lemma-Hest}.
\end{proof}

Now, we have prepared all ingredients to define 
$J^1\e$. It is as follows:
$$(J^1\e f)(x)=
{1\over\sqrt{\mu_{n-1}(\tilde S\e)}}\begin{cases}f(z),&x=(\tilde x,z)\in\Omega\e^-\text{ with }z<-\eps,\\
f(-\eps)-{1\over 2}\psi\e(x)\varphi\e(x)\big(f(-\eps)-f(\eps)\big),&x=(\tilde x,z)\in\Omega\e^-\text{ with }-\eps\leq z\le 0, \\
f(\eps)+{1\over 2}\psi\e(x)\varphi\e(x)\big(f(-\eps)-f(\eps)\big),&x=(\tilde x,z)\in\Omega\e^+\text{ with }0\le z\leq\eps,\\
f(z),&x=(\tilde x,z)\in\Omega\e^+\text{ with }z>\eps
\end{cases}
$$
(recall that $\mu_{n-1}(\cdot)$ denotes the Lebesgue measure of a set in $\R^{n-1}$).
\smallskip

Finally, we define the operator $\check J\e^{1}$ by
\begin{gather*}
\check J\e^{1}=J\e\restr{\HS^1\e}.
\end{gather*}

Since $(J\e)^*=\check J\e$,  condition \eqref{cond2} holds with $\delta\e= 0$. Thus, it remains to check conditions \eqref{cond1} and \eqref{cond3}.

\subsubsection{Proof of~\eqref{cond1}}

It is clear that the second inequality in \eqref{cond1} holds for each $\delta\e\geq 0$. Now, we prove the first inequality.

Let  $f\in \H(\Omega\setminus\{0\})$. In particular, this implies $f\in\mathsf{L}^\infty( \Omega)$, moreover one has the following standard Sobolev-type estimate:
\begin{gather}\label{sobolev}
\|f\| _{\mathsf{L}^\infty(\Omega)}\leq C\|f\|_{\H(\Omega\setminus\{0\})}.
\end{gather}

Since $0\leq \phi\e\leq 1$, $0\le\psi\e\leq 1$ (this follows from \emph{maximum principle} for harmonic functions),  one has 
\begin{gather}\label{J1est}
|(J^1 f)(x)|\leq
{1\over\sqrt{\mu_{n-1}(\tilde S\e)}}
\max\left\{|f(-\eps)|,\,|f(\eps)|\right\},\quad x\in Y\e.
\end{gather}
Then, using \eqref{equiv-norms}, \eqref{sobolev} and the fact that $\mu_n(Y\e)=2\eps\,\mu_{n-1}(\tilde S\e)$,  we obtain the desired estimate \eqref{cond1}:
\begin{multline}\label{cond1full}
\|J\e f - J\e^1 f\|_{\HS\e}= \|J\e f - J\e^1 f\|_{\L(Y\e)}\leq 
 \|J\e f\|_{\L(Y\e)}+\|J^1\e f\|_{\L(Y\e)} \\\leq
 \sqrt{\mu_n(Y\e)}\left(\|J\e f\|_{\mathsf{L}^\infty(Y\e)}+\|J^1\e f\|_{\mathsf{L}^\infty(Y\e)}\right) 
  \leq 2\sqrt{\mu_n(Y\e)\over\mu_{n-1}(\tilde S\e)} \|f\| _{\mathsf{L}^\infty(-\eps,\eps)}  \\ \leq 2\sqrt{2}\,{\eps}^{1/2}\|f\|_{\H(\Omega\setminus\{0\})}\leq 2\sqrt{2}\,\eps^{1/2}\|f\|_{\HS^1}.
\end{multline}

\subsubsection{Proof of~\eqref{cond3}}

Let $f\in  \mathrm{dom}(\A^\gamma)$, $u\in \H(\Omega\e)$.  
One has:
\begin{gather*}
\a\e(J^1\e f,u)-\a(f, \check J\e^{1} u) =
\mathcal{I}^1\e[f,u]+\mathcal{I}^2\e[f,u]+\mathcal{I}^3\e[f,u],
\end{gather*}
where
$$
\begin{array}{l}
\mathcal{I}^0\e[f,u]:=(\nabla J^1\e f,\nabla u)_{\L(\Omega\e\setminus\overline{Y\e})}-(f', (\check J\e^{1}u)')_{\L(\Omega\setminus[-\eps,\eps])} ,\\[1mm]
\mathcal{I}^1\e[f,u]:=-\, (f', (\check J\e^{1}u)')_{\L(-\eps,0)}-(f', (\check J\e^{1}u)')_{\L(0,\eps)} ,\\[1mm]
\mathcal{I}\e^2[f,u]:= (\nabla J^1\e f,\nabla u)_{\L(Y\e)} -\ds {\gamma\over 4}(f(+0)-f(-0))\overline{(\check J\e^{1}u(+0)-\check J\e^{1}u(-0))},\\[1mm]
\mathcal{I}^3\e[f,u]:=(V\e J^1\e f,u)_{\L(\Omega\e)}-(V f, \check J^1\e u)_{\L(\Omega)} .
\end{array}
$$
 
It is easy to see, that
\begin{gather}\label{I0-est}
\mathcal{I}\e^0[f,u]=0.
\end{gather} 

Let us estimate the term $\mathcal{I}^1\e[f,u]$. 
Since $f\in \mathsf{H}^2(\Omega\setminus\{0\})$,  
 $f'\in\mathsf{L}^\infty( \Omega^\pm)$ and  
$\|f'\|_{\mathsf{L}^\infty( \Omega^\pm)} \leq C\| {f}\|_{\mathsf{H}^2(\Omega\setminus\{0\})}$. Consequently, 
\begin{gather}\label{f-est}
\|f'\|_{\L((-\eps,\eps)\setminus\{0\})} \leq C\eps^{1/2}\| {f}\|_{\mathsf{H}^2(\Omega\setminus\{0\})}.
\end{gather}

Taking into account \eqref{scale+} and \eqref{equiv-norms},   one gets
\begin{multline}\label{H2est}
\|f\|^2_{\mathsf{H}^2(\Omega\setminus\{0\})}=
\|\A^\gamma f\|_{\L(\Omega)}^2+\|f\|_{\H(\Omega\setminus\{0\})}^2\leq 
2\|\A^\gamma f+ f\|_{\L(\Omega)}^2+2\|f \|_{\L(\Omega)}^2+\|f\|_{\H(\Omega\setminus\{0\})}^2
\\ \leq 
 2\|f\|^2_{\HS^2}+ 2\|f\|^2_{\HS}+\|f\|^2_{\HS^1} \leq 5\|f\|^2_{\HS^2}.
\end{multline}

Also, we notice that $ (\check J\e u)'=J\e {\partial u\over\partial z}$, $u\in \H(\Omega\e)$. Hence, employing   \eqref{CS-P1}, we obtain:
\begin{gather}\label{CS-P+}
\|(\check  J\e^1 u)'\|_{\L(\Omega^\pm)}=\|(\check  J\e u)'\|_{\L(\Omega^\pm)}=
\left\|\check  J\e {\partial u\over\partial z}\right\|_{\L(\Omega^\pm)}
\leq
\left\| {\partial u\over\partial z}\right\|_{\L(\Omega\e)}
\leq
\|u\|_{\HS\e^1}. 
\end{gather} 

Finally, using  \eqref{f-est}-\eqref{CS-P+}, we arrive at the estimate
\begin{multline}\label{I1-est}
|\mathcal{I}^1\e[f,u]|\leq \|f'\|_{\L(-\eps,0)}\|(\check J\e^1 u)'\|_{\L(-\eps,0)}+\|f'\|_{\L(0,\eps)}\|(\check J\e^1 u)'\|_{\L(0,\eps)}\leq
C\eps^{1/2} \|f\|_{\HS^2}\| u\|_{\HS^1\e}.
\end{multline}

Now, we start to inspect the term $\mathcal{I}^2\e[f,u]$.
We denote by  $\Delta_{Y\e}^N$  the Neumann Laplacian in $Y\e$.
It is easy to see that $(J^1\e f)\restr{Y\e}\in \mathrm{dom}(\Delta_{Y\e}^N)$, this follows from the form of $J^1\e f$ and the properties \eqref{psi-n1}-\eqref{psi-n2} of the function 
$\psi\e$.  
 
We denote:
$$u\e^\pm:=\ds{1\over \mu_n(Y^\pm\e)}\int_{Y^\pm\e} u(x)\d x.$$  
Integrating by parts one gets
\begin{multline*} 
(\nabla J^1\e f,\nabla u)_{\L(Y\e)}=-(\Delta_{Y\e}^N J^1\e f, u)_{\L(Y\e)}=-(\Delta J^1\e f, u^-\e)_{\L(Y^-\e)} -(\Delta J^1\e f, u^+\e)_{\L(Y^+\e)} + R\e[f,u],
\end{multline*}
where  $R\e[f,u]$ is the remainder term,
$$R\e[f,u]:=(\Delta J^1\e f, u^-\e - u)_{\L(Y^-\e)}  + (\Delta J^1\e f, u^+\e - u)_{\L(Y^+\e)}.$$
Then, again integrating by parts and using Lemma~\ref{another-cap}, we obtain
\begin{multline}\label{nablanabla}
(\nabla J^1\e f,\nabla u)_{\L(Y\e)}=
{f( \eps)-f(-\eps)\over 2\sqrt{\mu_{n-1}(\tilde S\e)}}\overline{\left(\overline{u^+\e}
\int_{D\e}\left.{\partial \psi\e\over\partial z}\right|_{z=+0} - \overline{u^-\e}
\int_{D\e}\left.{\partial \psi\e\over\partial z}\right|_{z=-0}\d\s \right)}+ R\e[f,u]\\=
{\gamma\e\over 4}\,(f( \eps)-f(-\eps))\, \sqrt{\mu_{n-1}(\tilde S\e)}\,\,\overline{\left(u^+\e-  u^-\e\right)}+R\e[f,u]
\end{multline}
(recall, that $\gamma\e$ is defined by \eqref{gamma-eps}).
It follows from \eqref{nablanabla} that $\mathcal{I}^2\e[f,u]$ can be represented in the form
\begin{gather}\label{PQR}
\mathcal{I}\e^2[f,u]=P\e[f,u]+Q\e[f,u]+R\e[f,u], 
\end{gather}
where
\begin{gather*}
\begin{array}{lll}
P\e[f,u]&:=&\ds
{1\over 4}\left(\gamma\e-\gamma\right)(f(+0)-f(-0))\overline{(\check J\e u(+0)-\check J\e u(-0))}  
,\\[3mm]
Q\e[f,u]&:=&\ds{\gamma\e\over 4}\,\left[(f( \eps)-f(-\eps)) \sqrt{\mu_{n-1}(\tilde S\e)}\,\,\overline{(u^-\e-u^+\e)}-(f(+0)-f(-0))\overline{(\check J\e u(+0)-\check J\e u(-0))}\right]
\end{array}
\end{gather*}

Let us estimate step-by-step the terms in the right-hand-side of \eqref{PQR}.

\begin{lemma}\label{Pe}One has:
\begin{gather}
\label{Pe-ineq}
\left|P\e[f,u] \right|\leq C\left|\gamma\e-\gamma \right|\|f\|_{\HS^1}\|u\|_{\HS^1\e}.
\end{gather}
\end{lemma}

\begin{proof} \eqref{Pe-ineq} follows  from 
the trace inequalities 
$
|f(\pm 0)|\leq C\|f\|_{\H(\Omega^\pm )},\ 
|\check J\e u(\pm 0)|\leq C\|\check J\e u\|_{\H(\Omega^\pm)} 
$ and \eqref{CS-P1}, \eqref{CS-P+}.  
\end{proof}

\begin{lemma}\label{Qe}
One has:
\begin{gather*}
|Q\e[f,u]|\leq   C\eps^{1/2}\|f\|_{\HS^1}\|u\|_{\HS^1\e}.
\end{gather*}
\end{lemma}

\begin{proof}Let us prove  the following estimate:
\begin{gather}\label{trace-poincare}
\forall u\in \H(Y\e^\pm):\quad \left| \check J\e u(\pm 0) - {u^\pm\e} \sqrt{\mu_{n-1}(\tilde S\e)}\right|\leq C\eps^{1/2}\|\nabla u\|_{\L(Y\e^\pm)}.
\end{gather}
At first we prove \eqref{trace-poincare} for $\eps=1$.
One has (in the estimate below we omit the subscript $\eps$ since it has fixed value $1$):
\begin{multline}\label{eps0}
\left|  \check J  u(\pm 0) - {u^\pm }\sqrt{\mu_{n-1}(\tilde S)}\right| 
=
 {1\over \sqrt{\mu_{n-1}(\tilde S)}}\left|\int_{\tilde S}\left(u(\tilde x,\pm 0)-{u^\pm }\right)\d\tilde x\right| 
\leq
 \|u -{u^\pm }\| _{\L(\tilde S)}\\\leq 
C\sqrt{\|u -{u^\pm }\|^2_{\L(Y ^\pm)}+\|\nabla u  \|^2_{\L(Y ^\pm)}}\leq C_1\|\nabla u  \| _{\L(Y ^\pm)}.
\end{multline}
Here the first estimate is the Cauchy-Schwarz inequality, the second one is the trace inequality and the third one is the Poincar\'{e} inequality. 
Since  
$$Y\e^\pm={\eps  } Y^\pm,\quad \tilde S\e={\eps  } \tilde S,\quad \eps >0,$$  for an  arbitrary $\eps$ \eqref{trace-poincare} follows from \eqref{eps0} via simple rescaling arguments.\smallskip

Now, using \eqref{trace-poincare} and taking into account that 
\begin{gather*}
|f(\pm 0)-f(\pm\eps)|\leq \eps^{1/2}\|f'\|_{\L(\Omega^\pm)},\quad |f(\pm \eps)|\leq C\|f\|_{\H(\Omega^\pm)},\\ |\check J\e u(\pm 0)|\leq C\|\check J\e u \|_{\H(\Omega^\pm)}\leq C\|u\|_{\H(\Omega\e^\pm)},
\end{gather*} 
we obtain the desired estimate.
\end{proof}

\begin{lemma}\label{Re}One has
\begin{gather*}
|R\e[f,u]|\leq C\|f\|_{\HS^1} \|u\|_{\HS^1\e}\begin{cases}
\eps^{1/2} ,&n\geq 3,
\\[3mm]
\eps^{1/2} |\ln \eps| ,&n=2.
\end{cases}
\end{gather*}
\end{lemma}

\begin{proof} 
Since $\Delta\psi\e=0$, we have 
$$(\Delta J^1\e f)\restr{Y^\pm\e}= 
\pm{1\over 2\sqrt{\mu_{n-1}(\tilde S\e)}}(f(-\eps)-f(\eps))(2\nabla\psi\e\cdot \nabla\phi\e + \psi\e \Delta\phi\e).$$ 

Let $G\e:=\{x:\ \nabla\phi\e(x)=0\}$. It is clear that $|x|\geq C\eps$ for $x\in G\e$. Using Lemma~\ref{lemma-Hest}, we get the  
pointwise estimate  
\begin{gather}\label{pointwise1}
|(\Delta J^1\e f)(x)|\leq 
C\eps^{1-3n\over 2}\|f\|_{\mathsf{L}^\infty(\Omega)}\begin{cases}
\ds{(d\e)^{n-2}},&n\geq 3,
\\[3mm]
\ds{|\ln \eps|\cdot |\ln d\e|^{-1}},&n=2,
\end{cases}\text{ as } x\in G\e,
\end{gather}
Moreover,
\begin{gather}
\label{pointwise1+}
\Delta J^1\e f=0\text{ as }  x\notin G\e.
\end{gather} 
 
Due to \eqref{gamma},  \eqref{cap-as1}, \eqref{cap-as2} (recall, that $\gamma<\infty$),
\begin{gather*}
(d\e)^{n-2}\leq C\eps^{n-1}\text{ as }n\geq 3;\quad |\ln d\e|^{-1}\leq C\eps\text{ as }n=2.
\end{gather*}
Then \eqref{pointwise1}-\eqref{pointwise1+} become
\begin{gather}\label{pointwise2}
|(\Delta J^1\e f)(x)|\leq 
C\|f\|_{\mathsf{L}^\infty(\Omega)}\begin{cases}
\eps^{-{1+n\over 2}},&n\geq 3,
\\[3mm]
\eps^{-{1+n\over 2}}|\ln \eps| ,&n=2.
\end{cases}
\end{gather}

We also have the following Poincar\'e inequality:
\begin{gather}\label{poincare}
\|u-u^\pm\e\|_{\L(Y^\pm\e)}\leq C\eps  \|\nabla u\| _{\L(Y\e^\pm)}.
\end{gather}
Combining \eqref{pointwise2} and \eqref{poincare}  we obtain:
\begin{multline*}
|(\Delta J^1\e f, u^\pm\e - u)_{\L(Y^\pm\e)}| \leq 
|Y^\pm\e|^{1/2}\cdot\|\Delta J^1\e f\|_{\mathsf{L}^\infty(Y^\pm\e)}\cdot \|u^\pm\e - u\|_{\L(Y^\pm\e)}\\ \leq 
C\|f\|_{\mathsf{L}^\infty(\Omega)}\|\nabla u\|_{\L(\Omega^\pm\e)}
\begin{cases}
\eps^{1/2} ,&n\geq 3,
\\[3mm]
\eps^{1/2} |\ln \eps| ,&n=2.
\end{cases}
\end{multline*}
The lemma is proved since 
$\|f\|_{\mathsf{L}^\infty(\Omega)}\leq \|f\|_{\HS^1}$ and $\|\nabla u\|_{\L(\Omega)}\leq \|u\|_{\HS\e^1}$.
\end{proof}

Lemmata~\ref{Pe}-\ref{Re} give   
\begin{gather}
\label{I2-est}
|\mathcal{I}^2\e[f,u]|\leq 
\begin{cases}
C\left(\eps^{1/2}|\ln\eps|+|\gamma\e-\gamma|\right) \|f\|_{\HS^1} \|u\|_{\HS^1\e},&n=2,\\
C\left(\eps^{1/2}+|\gamma\e-\gamma|\right) \|f\|_{\HS^1} \|u\|_{\HS^1\e},&n\ge 3.
\end{cases}
\end{gather} 
\medskip 

It remains to estimate the term $\mathcal{I}\e^3$. One has:
\begin{multline}\label{I3}
\mathcal{I}\e^3[f,u]=(V\e J\e f,u)_{\L(\Omega\e)}-(V f, \check J\e u)_{\L(\Omega)}\\+
(V\e (J^1\e f - J\e f),u)_{\L(\Omega\e)}+(V f, (\check J\e-\check J^1\e) u)_{\L(\Omega)}.
\end{multline}
Since $\check J^1\e=\check J\e\restriction_{\mathcal{H}\e^1}$, we get
\begin{gather}\label{I3a}
(V f, (\check J\e-\check J^1\e) u)_{\L(\Omega)}=0 
\end{gather} 
Using \eqref{supV} and   \eqref{cond1full} we obtain:
\begin{gather}\label{I3b}
|(V\e (J^1\e f - J\e f),u)_{\L(\Omega\e)}|\leq \|V\e\|_{\mathsf{L}^\infty(\Omega\e)}\|J^1\e f - J\e f\|_{\L(\Omega)}\|u\|_{\L(\Omega\e)}\leq C\eps^{1/2} \|f\|_{\mathcal{H}^1}\|u\|_{\mathcal{H}\e}.
\end{gather}
Finally, we examine the first two terms in \eqref{I3}. Using
\begin{gather}
\label{Iprop}
(\check J\e)^*=J\e,\quad J\e (Vf)=\sqrt{\mu_{n-1}(\tilde S\e)}\cdot J\e V\cdot J\e f
\end{gather}
and \eqref{CS} we get:
\begin{multline} \label{I3c}
\left|(V\e J\e f,u)_{\L(\Omega\e)}-(V f, \check J\e u)_{\L(\Omega)}\right|=
\left| \left((V\e -  \sqrt{\mu_{n-1}(\tilde S\e)} J\e V ) J\e f,  u\right)_{\L(\Omega\e)}\right|\\
\leq \left\|V\e -  \sqrt{\mu_{n-1}(\tilde S\e)} J\e V\right\|_{\mathsf{L}^\infty(\Omega\e)}
\| f\|_{\L(\Omega)}\| u \|_{\L(\Omega\e)}.
\end{multline}
Combining \eqref{I3a}, \eqref{I3b}, \eqref{I3c} we arrive at 
\begin{gather}\label{I3-est}
|\mathcal{I}^3\e[f,u]|\leq \left(C\eps^{1/2}+ \left\|V\e -  \sqrt{\mu_{n-1}(\tilde S\e)} J\e V\right\|_{\mathsf{L}^\infty(\Omega\e)}\right) \|f\|_{\HS^1} \|u\|_{\HS\e},
\end{gather}

Condition \eqref{cond3} follows from \eqref{I0-est}, \eqref{I1-est}, \eqref{I2-est}, \eqref{I3-est}.\medskip

Thus,  \eqref{cond1}-\eqref{cond3} hold true and, as we noticed above, this implies the fulfilment of 
Theorems~\ref{th1}--\ref{th2} for the case $\gamma<\infty$. In the next subsection we treat the case $\gamma=\infty$.

\subsection{The case $\gamma=\infty$\label{subsec42}}

In contrast to the case $\gamma<\infty$, now we define the operator $\check J\e^{1}$ using nontrivial expression, while $J\e^{1}$ is relatively simple:
\begin{gather*}
J\e^{1}=J\e\restr{\HS^1},\\[1mm]
\check J\e^{1} u(z)=
\begin{cases}
\check J\e  u(z),&z\in (L_-,L_+)\setminus (-\eps,\eps),\\
 {1\over 2\eps}\left(\check J\e  u(\eps)\cdot(z+\eps) - \check J\e  u(-\eps)\cdot(z-\eps)\right),&z\in [-\eps,\eps].
\end{cases}
\end{gather*}

Recall, that condition \eqref{cond2} holds with any $\delta\e\geq 0$ since $(J\e)^*=\check J\e$.
Thus, we only have to check conditions~\eqref{cond1} and~\eqref{cond3}.

The following lemma is essential for further considerations.

\begin{lemma}\label{lemma-diff}
Let $u\in \H(Y^\pm\e)$. Then the following inequlity holds:
\begin{gather*}
\left|{1\over \mu_{n-1}(\tilde S\e)}\int_{\tilde S\e}u(\tilde x,\pm 0)\d \tilde x - {1\over \mu_{n-1}(\tilde D\e)}\int_{\tilde D\e}u(\tilde x,0)\d \tilde x \right|\leq C\|\nabla u\| _{\L(Y\e^\pm)}\cdot
\begin{cases}
(d\e)^{2-n\over 2},& n\ge 3,\\
|\ln d\e|^{1/2},& n=2.
\end{cases}
\end{gather*}
\end{lemma}

The proof this lemma is similar to the proof of \citep[Lemma~2.1.I]{K13}.

\subsubsection{Proof of~\eqref{cond1}}
We have to check only the second inequality in \eqref{cond1}. The first one holds since its left-hand-side equals zero.

Let  $u\in \H(\Omega\e)$. It is clear that 
$$ 
\|\check J\e^{1}u\|_{\mathsf{L}^\infty(-\eps,\eps)}\leq
\max\{|\check J\e u(-\eps)|,\,|\check J\e u(\eps)|\}\leq
\|\check J\e u \|_{\mathsf{L}^\infty(-\eps,\eps)}\leq C\|\check J\e u\|_{\H(\Omega)}
.$$ Using this and taking   into account \eqref{CS-P1}, \eqref{CS-P+} we obtain:
\begin{multline}\label{cond1full+}
\|\check  J\e u-\check J\e^{1} u\|_{\HS}\leq \| \check J\e  u\|_{\L(-\eps,\eps)} +
\|\check J\e^{1} u\|_{\L(-\eps,\eps)}
\leq C\eps^{1/2}\left(\|\check J\e  u\|_{\mathsf{L}^\infty(-\eps,\eps)}+\|\check J\e^{1} u\|_{\mathsf{L}^\infty(-\eps,\eps)}\right)
\\\leq  C_1 \eps^{1/2}\|\check J\e u\|_{\H(\Omega)} \leq C_2\eps^{1/2}\|u\|_{\H(\Omega\e)}\leq C_3\eps^{1/2}\|u\|_{\HS^1\e}
\end{multline} 
and we arrive at the desired condition \eqref{cond1}.

\subsubsection{Proof of~\eqref{cond3}} 
Let $f\in  \mathrm{dom}(\A^\infty)$, $u\in \H(\Omega\e)$.  
One has:
\begin{multline*} 
 \a\e(J^1\e f,u)-\a(f, \check J\e^{1} u) 
\\=
\underset{\mathcal{I}^{1,-}\e:=}{\underbrace{\int_{\Omega\e^-}\nabla(J\e f)\cdot\overline{\nabla u}\d x - \int_{\Omega^-}f'\cdot\overline{(\check J\e u)'}\d z}} +
\underset{\mathcal{I}^{1,+}\e:=}{\underbrace{\int_{\Omega\e^+}\nabla(J\e f)\cdot\overline{\nabla u}\d x - \int_{\Omega^+}f'\cdot\overline{(\check J\e u)'}\d z}}\\+
\underset{\mathcal{I}^{2}\e:=}{\underbrace{\int_{\Omega }f'\cdot\overline{(\check J\e u - \check J^1\e u)'}\d  z}}+\underset{\mathcal{I}^{3}\e:=}{\underbrace{\int_{\Omega\e} V\e\cdot J\e^1 f\cdot \overline{u}\d x-
\int_{\Omega}V\cdot f \cdot\overline{\check J^1\e u}   \d  z}}.
\end{multline*}
 
It is clear that 
\begin{gather*}
\mathcal{I}^{1,-}\e=\mathcal{I}^{1,+}\e=0. 
\end{gather*}
Integrating by parts  we get:
\begin{gather*}
\mathcal{I}^{2}\e=-\int_{-\eps}^\eps f''\cdot\overline{(\check J\e u - \check J^1\e u)}\d z - f'(0)\cdot\overline{(\check J\e u( +0) - \check J\e u(-0))},
\end{gather*}
and thus, due to the trace inequality $|f'(0)|\leq C\|f'\|_{\mathsf{H}^1(\Omega)}$, 
\begin{gather}\label{est1}
|\mathcal{I}^{2}\e|\leq \|f''\|_{\L(\Omega)} \cdot \|\check J\e u - \check J^1\e u\|_{\L(\Omega)}+
C\|f'\|_{\H(\Omega)}\cdot\left|\check J\e u(+0) - \check J \e u(-0)\right|.
\end{gather}
The first term in \eqref{est1} can be estimated using \eqref{cond1full+}: 
\begin{gather}\label{first-term}
\|f''\|_{\L(\Omega)} \cdot\|\check J\e u - \check J^1\e u\|_{\L(\Omega)}\leq
C\eps^{1/2}\|f\|_{\HS^2}\|u\|_{\HS^1\e}.
\end{gather}
Now, we estimate the second term.
Using Lemma~\ref{lemma-diff} and taking into account \eqref{cap-as1}-\eqref{cap-as2}, we get
\begin{multline}\label{JJ2}
\left|\check  J\e u (+0)-\check  J\e u (-0)\right|=
{ \sqrt{\mu_{n-1}(\tilde S\e)}}\left|
{1\over  {\mu_{n-1}(\tilde S\e)}}\int_{\tilde S\e}u(\tilde x,0)\d \tilde x-
{1\over  {\mu_{n-1}(\tilde S\e)}}
\int_{\tilde S\e}u(\tilde x,-0)\d \tilde x
\right|\\
\leq C(\gamma\e)^{-1/2}\|\nabla u\|_{\L(Y\e)}.
\end{multline}

Combining \eqref{est1}-\eqref{JJ2} we obtain:
\begin{gather}\label{I2est-infty}
|\mathcal{I}^{2}\e|\leq C\left((\gamma\e)^{-1/2}+\eps^{1/2}\right)\,\|f\|_{\HS^2} \cdot \|u\|_{\HS^1}.
\end{gather}

Finally, we estimate the term $\mathcal{I}^{3}\e$. 
Using  \eqref{CS}, \eqref{Iprop}, \eqref{cond1full+}  and taking into account that $J\e^1=J\e\restriction_{\HS^1}$ we obtain
\begin{multline}\label{I3est-infty}
|\mathcal{I}^3\e|=\left|(V\e J\e f,u)_{\L(\Omega\e)}-(V f, \check J\e u)_{\L(\Omega)}+
(V f, (\check J \e-\check J^1\e) u)_{\L(\Omega)}\right|\\=
\left|\big((V\e -  \sqrt{\mu_{n-1}(\tilde S\e)} J\e V ) J\e f,  u\big)_{\L(\Omega\e)}+
(V f, (\check J \e-\check J^1\e) u)_{\L(\Omega)}\right|
\\\leq 
\left\|V\e -  \sqrt{\mu_{n-1}(\tilde S\e)} J\e V\right\|_{\mathsf{L}^\infty(\Omega\e)}
\|J\e f\|_{\L(\Omega\e)}\|u\|_{\L(\Omega\e)}\\+
\|V\|_{\mathsf{L}^\infty(\Omega)}\|f\|_{\L(\Omega)}\|(\check J^1\e-\check J\e) u\|_{\L(\Omega)}\leq
\left( \left\|V\e -  \sqrt{\mu_{n-1}(\tilde S\e)} J\e V\right\|_{\mathsf{L}^\infty(\Omega\e)}+C\eps^{1/2}\right)\|f\|_{\HS}\|u\|_{\HS^1\e}.
\end{multline}
Estimates \eqref{I2est-infty}-\eqref{I3est-infty} imply \eqref{cond3}. 

\subsection{Proof of Theorem~\ref{th3}\label{subsec43}}

To prove Theorem~\ref{th3} we
 apply Theorem \ref{thm:IOS} with $B\e:=(\A\e+\Id)^{-1}$, $B:=(\A^\gamma+\Id)^{-1}$.
These operators are non-negative, self-adjoint and compact (the last property is due to $L_->-\infty,\ L_+<\infty$). Moreover
$\|B\e\|\leq 1$, thus condition~$A_2$ of Theorem \ref{thm:IOS} is
fulfilled. In fact, $B\e$ are not only bounded uniformly in $\eps$ as   operators from $\L(\Omega\e)$ to $\L(\Omega\e)$, but one has even stronger property, namely
\begin{gather}\label{A2+}
\|B\e f\|_{\H(\Omega\e)}\leq\|f\|_{\L(\Omega\e)}.
\end{gather}

Again we define the operator $J\e$ by \eqref{J}. 
Due to~\eqref{CS}, condition~$A_1$ of Theorem \ref{thm:IOS} is valid. Moreover, Theorem~\ref{th1} implies the fulfilment of  condition~$A_3$. 

It remains to check condition $A_4$. Let $\{f\e\in \L(\Omega\e)\}_\eps$ be a family with $\sup_{\eps}
\|f\e\|_{\L(\Omega\e)}<\infty$.
Then the family $\{\|B\e
f\e\|_{\H(\Omega\e)}\}_\eps$ is  bounded due to \eqref{A2+}.  
One has also $\|\check J\e u\|_{\H(\Omega\setminus\{0\})}\leq \|u\|_{\H(\Omega\e)}$, see~\eqref{CS-P1},\eqref{CS-P+} (recall that the
operator $\check J\e$ is defined by~\eqref{J'}). Then
the sequence
$\{\check J\e B\e f\e\}_\eps$ is also bounded in $\H(\Omega\setminus\{0\})$.
Using Rellich's embedding
theorem we conclude that
there exist $w\in \H(\Omega\setminus\{0\})$ and a sequence $(\eps_m)_{m\in\N}$ such
that 
\begin{gather}\label{A4+}
\|\check J_{\eps_m} B_{\eps_m} f_{\eps_m}-w\|_{\L(\Omega)}\to 0\text{ as
 }\eps_m\to 0.
\end{gather}

One has:
\begin{multline}\label{A4-main}
\|  B_{\eps_m}  f_{\eps_m}-J_{\eps_m}w\|_{\L(\Omega_{\eps_m})}\leq 
\| B_{\eps_m}  f_{\eps_m} - J_{\eps_m}\check J_{\eps_m}  B_{\eps_m}  f_{\eps_m} \|_{\L(\Omega_{\eps_m})}+
\| J_{\eps_m}\check J_{\eps_m}  B_{\eps_m}  f_{\eps_m}-J_{\eps_m}w\|_{\L(\Omega_{\eps_m})}
\\\leq\|\Id - J_{\eps_m}\check J_{\eps_m}\|_{\H(\Omega_{\eps_m})\to \L(\Omega_{\eps_m})}\cdot\|B_{\eps_m}\|_{\L(\Omega_{\eps_m})\to \H(\Omega_{\eps_m})}\cdot \|f_{\eps_m}\|_{\L(\Omega_{\eps_m})}\, +\,\|J_{\eps_m}\|_{\L(\Omega )\to \L(\Omega_{\eps_m})}\cdot \|\check J_{\eps_m} B_{\eps_m} f_{\eps_m}-w\|_{\L(\Omega)}.
\end{multline}

Finally, for any  $u\in \H(\Omega\e)$ we get, using the Poincar\'e inequality on $\tilde S\e$,  
\begin{multline*}
\|u - J\e \check J\e u\|^2_{\L(\Omega\e)}= \int_{\Omega_-}\int_{\tilde S\e}\left|u(\tilde x,z)-{1\over \mu_{n-1}(\tilde S\e)}\int_{\tilde S\e}u(\tilde y,z)\d \tilde y\right|^2\d \tilde x \d z\\\leq C\eps^2\int_{\Omega }\int_{\tilde S\e}|\nabla_{\tilde x} u|^2\d \tilde x\d z\leq  C  \eps^2 \|u\|^2_{\H(\Omega\e)},
\end{multline*} 
and thus 
\begin{gather}\label{cond-add}
\|\Id - J_{\eps_m}\check J_{\eps_m}\|_{\H(\Omega\e)\to \L(\Omega\e)}\to 0.
\end{gather}
Condition $A_4$ follows immediately from \eqref{CS}, \eqref{A2+}-\eqref{cond-add}.
\smallskip

Combining Theorems~\ref{th1},\ref{thm:IOS} and already proved Theorem~\ref{th1} we arrive at the
estimate \eqref{spectrum2}. 
Evidently, \eqref{spectrum1} follows from \eqref{spectrum2}.
Theorem~\ref{th3} is proven.

\section*{Acknowledgment}

A.K. is supported by the Austrian Science Fund (FWF) under Project No. M~2310-N32.
G.C. is a member of GNAMPA (INdAM). The research was also supported by grant FFABR 2017 (Finanziamento annuale individuale delle attività base di ricerca).

\end{document}